\documentclass[final,1p,times,twocolumn]{elsarticle}

\usepackage{amssymb}
\usepackage{amsmath}
\usepackage{graphicx}
\usepackage[mathscr]{euscript}
\usepackage{amsthm}
\usepackage{cleveref}
\usepackage{pgfplots}
\usepackage{wasysym}
\usepackage{subcaption}

\theoremstyle{definition}
\newtheorem{dfn}{Definition}[section]
\numberwithin{dfn}{section}
\theoremstyle{theorem}
\newtheorem{prop}[dfn]{Proposition}
\theoremstyle{definition}

\numberwithin{equation}{section}

\newcommand\bb[1]{\mathbb{#1}}

\newcommand{\dds}{\frac{d}{ds}}

\newcommand{\ind}{\textup{ind}}
\newcommand\m[1]{\begin{pmatrix}#1\end{pmatrix}}

\journal{Journal of Geometry and Physics}

\begin{document}
	
	\begin{frontmatter}
		
		
		
		\title{Morse Theory and Stability of Periodic Billiard Trajectories} 

\author[Henry, Dan]{Henry Kavle, Daniel Offin} 

\affiliation[Henry]{organization={Norwich University},
		city={Northfield},
		postcode={05663}, 
		state={Vermont},
		country={United States}}
\affiliation[Dan]{organization={Department of Mathematics and Statistics, Queen's University at Kingston},
		addressline={}, 
		city={Kingston},
		postcode={K7L 3N6}, 
		state={Ontario},
		country={Canada}}

		
		\begin{abstract}
			In a billiard table with $C^2$ boundary, trajectories are critical points of the negative length functional on the space of admissible paths in the billiard table. The goal of this paper is twofold: first, to prove theorems which aid in computing the Morse index of trajectories; second, to use the Morse index to determine the stability of periodic trajectories. To the first end, we prove discrete analogs to the Morse index theorem for arbitrary trajectories, which are critical points with respect to fixed-endpoint variations, and for periodic trajectories, which are critical points with respect to periodic variations. To the second end, we prove a criterion for linear stability of periodic trajectories depending on the Morse index of the second iterate of the trajectory, and prove a sufficient condition for a billiard table to carry a linearly stable 2-periodic trajectory. 
		\end{abstract}
		
		
		

	\end{frontmatter}
	
		
		
\section{Introduction and Background}
		
\subsection{Billiards}
Let $\gamma:S^1 \to \bb{R}^2$ be a $C^2$ plane curve; this will be the boundary of our billiard table. Suppose for the sake of computation that $\gamma$ is positively oriented and its parametrization is unit-speed. A length-$N$ trajectory of a billiard is determined by a sequence of $N+1$ reflection points $x^i \in S^1$. Associated to this is a sequence of reflection angles $\phi^i \in (0,\pi)$, measured between the outgoing billiard path at the point $x^i$ and the velocity vector $\dot{\gamma}(x^i)$, and $\bar{\phi}^i \in (0,\pi)$, measured between the incoming billiard path at the point $x^i$ and the velocity vector $\dot{\gamma}(x^i)$. With this notation, the elastic reflection law is $\phi^i = \bar{\phi}^i$, that is, ``the angle of reflection equals the angle of incidence," see \cref{billiardNotation}.
		
\begin{figure}
	\centering
	\begin{tikzpicture}
		\draw [->] plot [smooth,tension=1] coordinates{(-2.5,.5) (0,0) (2.5,.5)};
		\draw [->] (0,0)--(2,0);
		\draw [>-] (-1,1)--(0,0);
		\draw [dashed] (-1.5,1.5)--(-1,1);
		\draw [->] (0,0) --(2,2);
		\draw[dashed] [->](0,0)--(1,-1);
		\draw (1/2,0) arc [radius = 1/2,start angle = 0, end angle = -45];
		\draw (1/2,0) arc [radius = 1/2,start angle = 0, end angle = 45];
		\node at (.8,1/2) [anchor = west] {$\phi^{i+1}$};
		\node at (.8,-1/2)[anchor = west] {$\bar{\phi}^{i+1}$};
		\node at (-1,-1/2) [anchor = west] {$x^{i+1}$};
		\node at (-1.6,1.6) [anchor = south]{$x^i,\phi^i$};
		\node at (2.5,.5) [anchor = west] {$\gamma$};
		\node at (2,0) [anchor = west] {$\dot{\gamma}$};
	\end{tikzpicture}
	\caption{\label{billiardNotation}Notation for the billiard system. }
\end{figure}
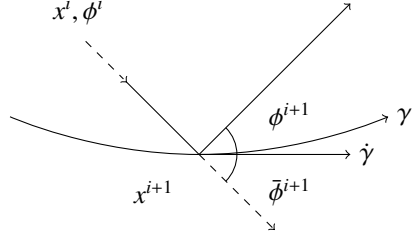
		
The elastic reflection law is equivalent to a variational principle; trajectories of the billiard map extremize their path length. Let $\mathcal{C}^N \subset (S^1)^{N+1}$ be the space of admissible paths, that is, those such that the line segment between $\gamma(x^i)$ and $\gamma(x^{i+1})$ is contained inside and never tangent to $\gamma$ for all $i = 1,...,N+1$. Note that when $\gamma$ is convex, $\mathcal{C}^N = (S^1)^{N+1}$. When $\gamma$ is not convex, $\mathcal{C}^N$ is an open submanifold of $(S^1)^{N+1}$. Let $\mathcal{C}^N_R$ be the submanifold $x = \{(x^1,...,x^{N+1}) \in \mathcal{C}^N| x^i \neq x^{i+1}, \ i = 2,...,n\}$ satisfying some boundary relation $(x^1, x^{N+1}) \in R \subset S^1\times S^1$. For example, fixing a particular pair $R = (x^1,x^{N+1})$ specifies fixed-endpoint boundary conditions, whereas $R = \Delta$, the diagonal in $S^1\times S^1$, specifies periodic boundary conditions. 
\begin{dfn} The space of variations with respect to a boundary relation $R$ is $T_x\mathcal{C}^N_R$. We reserve the notation $\mathcal{C}_0^N$ for paths with fixed endpoints and $T_x\mathcal{C}^N_0$ for fixed-endpoint variations. We reserve the notation $\mathcal{C}_\Delta^N$ for periodic paths and $T_x\mathcal{C}^N_\Delta$ for periodic variations. 
\end{dfn}
Our Lagrangian is the negative length functional
\begin{equation}\label{billiardLengthFunctional}
	L(x) = -\sum_{i=1}^N \ell(x^i,x^{i+1})
\end{equation}
where $\ell(x^i,x^{i+1})$ is the Euclidean distance between $\gamma(x^i)$ and $\gamma(x^{i+1})$. As shorthand, we denote $\ell(x^i,x^{i+1}) = \ell_i.$ In order to clean up signs, we may also use $S_i = -\ell_i.$ Critical points of this functional satisfy the discrete Euler-Lagrange equation
\begin{equation}\label{discreteE-L}
	\partial_1\ell_i + \partial_2\ell_{i-1} = 0, \ i = 2,...,N
\end{equation}
subject to boundary conditions
\begin{equation}\label{discreteBCs}
	\partial_1\ell_1\delta x^1 + \partial_2\ell_{N} \delta x^{N+1}  = 0 
\end{equation}
for all variations $(\delta x^1, \delta x^{N+1}) \in TR$. 
		
A computation shows that $-\partial_1(-\ell_i) = -\cos(\phi^i)$ and $\partial_2(-\ell_{i-1}) = -\cos(\bar{\phi}^i)$, so the discrete Euler-Lagrange equations are equivalent to the elastic reflection law. Furthermore, the mixed partial $\partial^2_{21}(-\ell_i)$ is always negative. We will compute this later, but intuitively, for an admissible trajectory, fixing $x^i$ and increasing $x^{i+1}$ increases $\phi^i$, and for $\phi^i \in (0,\pi)$, $\cos(\phi^i)$ is decreasing. This also ensures that given any consecutive $x^{i-1}$ and $x^i$, there is a unique $x^{i+1}$ satisfying the discrete Euler-Lagrange equation.
		
The boundary conditions \cref{discreteBCs} on the partials are dependent on the boundary relation specified by $R$. Abstractly, our critical point must satisfy $dL = 0$ on $T_x\mathcal{C}^N_R$, although for the sake of working in convenient coordinates, we handle this as a subspace of $T_x\mathcal{C}^N$. For instance, with the periodic boundary relation $R = \Delta$, the boundary condition on the partials ensures $\phi^1 = \phi^{N+1}$. 
		
As suggested above, we may also describe a billiard as a discrete Hamiltonian system. Formally, $-\ell$ is a local generating function for the symplectic map 
\begin{align}
	T:& S^1 \times [-1,1] \to S^1 \times [-1,1];\\
	& \m{x^i \\y^i} \mapsto \m{x^{i+1}\\y^{i+1}}
\end{align}
where $y^i = -\partial_1(-\ell(x^i,x^{i+1})) = -\cos(\phi^i)$ If $\gamma$ is strictly convex, $-\ell$ is a global generating function for an exact monotone symplectic twist map of the annulus. If $\gamma$ is not strictly convex, the billiard mapping on the phase annulus is a local symplectomorphism in a neighborhood of an admissible trajectory, again with $-\ell$ as a local generating function, but it is no longer a symplectomorphism of the annulus, as it is discontinuous. For details and other analysis, see \cite{Tabachnikov:billiards}, \cite{McD-S:IntroSympTop}, and \cite{BialyTsod}.  We will focus primarily on the Lagrangian viewpoint, but the Hamiltonian setting is necessary to discuss stability of periodic trajectories in a dynamical systems sense.
		
		\subsection{Pieces of Morse Theory}
		Here we recall key definitions in Morse theory and state a technical proposition that will help us compute the indices of trajectories in billiards. 
		
		\begin{dfn} Let $H$ be a symmetric bilinear form on a finite-dimensional vector space $U$. The \emph{index} of $H$ is the maximal dimension of a subspace of $U$ on which $H$ is negative definite. The \emph{nullity} of $H$ is the dimension of the nullspace of $H$. 
		\end{dfn}
		
		\begin{dfn} Let $M$ be a smooth manifold, $f:M\to \bb{R}$ a smooth function, and $x \in M$ a critical point of $f$. The \emph{Hessian} of $f$ at $x$ is the symmetric bilinear form $f_{**}: T_xM \times T_xM \to \bb{R}$ defined by $f_{**}(V,W) = \tilde{V}(\tilde{W}(f))$, where $\tilde{V}, \ \tilde{W}$ are extensions of $V,W$ to vector fields on $M$. 
		\end{dfn} 
		
		\begin{dfn} Let $M$ be a smooth manifold, $f:M\to \bb{R}$ a smooth function, and $x \in M$ a critical point of $f$. The \emph{Morse index} $i(x)$ of $f$ at $x$ is the index of the Hessian of $f$ at $x$. Likewise, the \emph{nullity} is the dimension of the nullspace of the Hessian $f$ at $x$. 
		\end{dfn}
		\begin{dfn}A critical point is \emph{nondegenerate} if its nullity is zero, and a function is a \emph{Morse function} if it has only nondegenerate critical points.
		\end{dfn}  
		\noindent Note that the term ``index" refers to both the index of a bilinear form, and the index of a critical point. In practice the distinction is clear from context.
		
		In coordinates $\frac{\partial}{\partial x^i}$ on $T_xM$, the Hessian $f_{**}$ can be represented by the matrix $(D^2f_x)_{i,j} = \left(\frac{\partial^2 f}{\partial x^i\partial x^j}(x)\right)$, which we will denote by $H_x$. For the purposes of this paper, we will mostly work with this matrix expression.
		
		It is convenient to compute the Morse index on particular subspaces, rather than directly. We can use the following formula from \cite{BTZ:geodesics} to decompose the index computation into computations of indices on subspaces.
		\begin{prop}\label{indexSubspace} Let $H$ be a symmetric form on a finite dimensional vector space $U$. For any subspace $W \subset U$, 
			\begin{equation}\label{IndexDecomposition}
				\ind H = \ind H|_W + \ind H|_{W^\perp} + \dim (W \cap W^\perp) - \dim (W \cap \ker H),
			\end{equation}
			where $W^\perp$ is the orthogonal complement to $W$ with respect to the form $H$ and the kernel of $H$ is 
			\begin{equation}
				\ker H = \{ u \in U \ | \ H(u, v) = 0 \textrm{ for all } v \in U\}.
			\end{equation} 
		\end{prop}
		\noindent For a full proof, \cite{BTZ:geodesics} refers to Theorem 3.8 in \cite{Artin:GeometricAlgebra}.  
		
		\subsection{The Mountain Pass Theorem} 
		The mountain pass theorem, in its many variations, is a useful tool for finding additional critical points from known minimizing critical points. We state the standard version proved in \cite{Nicolaescu:MorseTheory} as well as a modified version for manifolds with boundary. 
		\begin{prop}[Mountain Pass Theorem for Compact Manifolds]\label{mtnPass} Let $M$ be a compact manifold, $f:M\to \bb{R}$ a smooth function. Suppose $x_0$ and $x_1$ are distinct local minima of $f$. Let $\Gamma = \{\gamma \in C^1([0,1],M) | \gamma(0) = x_0,\ \gamma(1) = x_1\}$. Then $f$ has a critical value $c$, attained by some critical point $x_2$, characterized by 
			\begin{equation} c = \inf_{\gamma \in \Gamma}\max_{s \in [0,1]} f(\gamma(s)).
			\end{equation}
		\end{prop}

\begin{prop}[Mountain Pass Theorem for Compact Manifolds with Boundary]\label{mtnPassBoundary} Let $M$ be a compact manifold with boundary, and $f:M\to R$ a smooth function. Suppose that the gradient $\nabla f$ given by an auxiliary Riemannian metric on $M$, is inward pointing along the boundary $\partial M$. Further suppose $x_0$ and $x_1$ are distinct local minima of $f$. Let $\Gamma = \{\gamma \in C^1([0,1],M) | \gamma(0) = x_0,\ \gamma(1) = x_1\}$. Then $f$ has a critical value $c$, attained by some critical point $x_2$, characterized by 
	\begin{equation} c = \inf_{\gamma \in \Gamma}\max_{s \in [0,1]} f(\gamma(s)).
	\end{equation}
\end{prop}
The hypothesis that the gradient point inward at the boundary is sufficient to proceed according to the proof in \cite{Nicolaescu:MorseTheory}.
\begin{dfn} We call the critical point $x_2$ from \cref {mtnPass} or \cref{mtnPassBoundary} a \emph{mountain pass type critical point}. 
\end{dfn} 

Assuming the critical point is nondegenerate, we automatically know the Morse index of a mountain pass type critical point. 
		\begin{prop}\label{mtnPassIndex} Let $M$ be a compact manifold, and let $x$ be a nondegenerate mountain pass type critical point of a Morse function $f:M \to \bb{R}$. Then $i(x) = 1$. 
		\end{prop}
		This is proved via a computation in the coordinates provided by the Morse lemma \cite{Milnor:morseindex}. 
		
		\subsection{Discrete Sturm-Liouville Theory}
		In the calculus of variations, the second variation is a Sturm-Liouville operator on the space of variations about a critical point, and one may relate the index of the second variation at the critical point to the number of zeros of solutions to this Sturm-Liouville problem; this is in essence the Morse index theorem \cite{Duistermaat:MorseIndex}. Our goal is to analyze a finite-dimensional variational problem in a similar fashion. Instead of a Sturm-Liouville problem, we will regard the second variation at the critical point as a second-order difference equation. The purpose of this section is to gather necessary results to relate sign changes of a solution to this difference equation to the index of the second variation. Most of these results are known within the theory of Jacobi operators \cite{Teschl:JacobiOperators}, and Bialy and Tsodikovich have set up a similar framework to analyze specifically local maxima of an alternative generating function \cite{BialyTsod}, but for clarity we will prove what we need from scratch.
		
		First, let us fix some definitions and notation. We will restrict our attention to boundary value problems for second order difference systems on finite sequences; let $V =\bb{R}^{N+1}$ be our space of length-$(N+1)$ real valued sequences. We will identify the sequence $v^i, \ i = 1,...,N+1$ with the vector $v$, and interpret $v$ as a sequence or vector in context as needed. 
		
		Let $\mathcal{J}: V \to V$ be an operator given by 
		\begin{align}\label{secondDifferenceOperator}
			(\mathcal{J}(v))^i = \begin{cases}a_1 v^1 + b_1 v^2, \ i = 1\\
				b_{i-1}v^{i-1} + a_i v^i + b_{i} v^{i+1}, \ i = 2,...,N\\
				b_{N}v^N + a_{N+1}v^{N+1}, \ i= N+1
			\end{cases}
		\end{align}
		with $a_i, b_i \in \bb{R}$ and $b_i < 0$. To be explicit about the connection to Sturm-Liouville operators, we note that if $\Delta^+$ and $\Delta^-$ denote forward and backwards difference operators
		\begin{align}
			\Delta^+(v)^i = v^{i+1}-v^i\\
			\Delta^-(v)^i = v^i - v^{i-1}
		\end{align}
		then we may write
		\begin{align}\mathcal{J}(v)^i = \begin{cases}
				b_1\Delta^+(v)^1 + (b_1+ a_1)v^1, \ i=1\\
				\Delta^-(b_i\Delta^+(v)^i) + (b_i + b_{i-1} + a_i)v^i, \  i= 2,...,N, \text{ or }\\
				\Delta^+(b_{i-1}\Delta^-(v)^i) + (b_i + b_{i-1} + a_i)v^i, i = 2,...,N\\
				-b_N\Delta^-(v)^{N+1} + (b_N + a_{N+1})v^{N+1}, \ i = n+1
			\end{cases}
		\end{align}   
		The boundary terms $\mathcal{J}(v)^1$ and $\mathcal{J}(v)^{N+1}$ are first-order differences which arise as a consequence of working with finite sequences, and the interior terms are second-order differences. We may identify the difference operator $\mathcal{J}$ with a symmetric, tridiagonal matrix
		\begin{equation}
			J = \m{a_1 & b_1 & 0 &...&0\\b_1&a_2&\ddots& \ &\vdots\\ 0&\ddots&\ddots&\ddots&0\\ \vdots & \ & \ddots & \ddots &b_N\\0&...&0&b_N&a_{N+1}},
		\end{equation}
		so that $\mathcal{J}(v^1,...,v^{N+1}) = Jv$. Such matrices are referred to as Jacobi matrices \cite{Teschl:JacobiOperators}.
		
		Before discussing boundary value problems, we note that given any $v^i, v^{i+1}$, we can solve the homogeneous initial value problem 
		\begin{equation}
			\mathcal{J}(v) = (c^1,0,...,0,c^2)
		\end{equation}
		via the forward recurrence 
		\begin{equation}\label{forwardRecurrence} v^{i+1} =\frac{-1}{b_i}(a_i v^i + b_{i-1}v^{i-1});
		\end{equation}
		or backward recurrence
		\begin{equation}\label{backwardRecurrence} v^{i-1} =\frac{-1}{b_{i-1}}(a_iv^i + b_iv^{i+1});
		\end{equation}
		Note that we have left $c^1$ and $c^2$ as variables even though we refer to a homogeneous initial value problem; they result from the fact that our operator $\mathcal{J}$ is a first-order difference at the endpoints. If consecutive values $v^i$ and $v^{i+1}$ for any $i = 1,...,N-1$ are specified, they fix the free variables $c^1$ and $c^2$. That is, providing consecutive values is analogous to imposing initial conditions in a second-order differential equation. Similarly, specifying boundary values $v^1$ and $v_{N+1}$ while leaving $c_1$ and $c_2$ free is analogous to imposing Dirichlet boundary conditions. Specifying $c_1$ and $c_2$ themselves is analogous to imposing Robin boundary conditions. 

We are particularly interested in periodic boundary conditions of the form 
\begin{align}\label{periodicBCs}
	v^1 = v^{N+1}.
\end{align}
Since the system is linear, any solution to this problem is a combination of solutions to the Dirichlet problems with $v^1 = v^{N+1} = 0$ and $v^1 = v^{N+1} = 1.$

Now we will set up a discrete Sturm-Liouville eigenvalue problem. We introduce an operator $\mathcal{Q}: V \to V$ defined by 
\begin{align} (\mathcal{Q}(v))^i =  \begin{cases} \frac{q}{2} v^i, i = 1, N+1\\
		qv^i, i = 2,...,N\end{cases},
\end{align} 
whose corresponding matrix $Q$ is a diagonal matrix with entries $(q/2,q,...,q,q/2)$ on the diagonal, for some $q>0$. We then introduce an eigenvalue problem 
\begin{align}\label{discreteSLEP}(\mathcal{J} - \mu\mathcal{Q})(v)^i = \begin{cases}
		c^1, \ i = 1\\
		0, \ i = 2,...,N\\
		c^2, \ i = N+1
	\end{cases} 
\end{align}
where $c^1$ and $c^2$ are free variables. We make two additional comments: First, the factors of 1/2 in the first and last entries of the operator $\mathcal{Q}$ make our use of the term ``eigenvalue'' different from the usual linear algebra sense, but allow us to extend our eigenvalue problem periodically to longer sequences in later applications. Second, it appears we have two eigen-parameters in our problem, $\mu$ and $q$; in practice, $q$ is fixed prior to setting up the eigenvalue problem, and $\mu$ will be regarded as the eigenvalue parameter. 

To conclude, we summarize all the above information into the following definition of a discrete Sturm-Liouville eigenvalue problem. 
\begin{dfn} We will call a second-order difference system \cref{discreteSLEP} with appropriate (\emph{e.g.} \cref{periodicBCs}) boundary conditions a \emph{discrete Sturm-Liouville problem}.
\end{dfn}

Our first result is a discrete analog to the Sturm-Picone theorem; we want to compare the number of sign changes in the entries of eigenvectors with different eigenvalues. We identify $v$ with the piecewise linear function $v: [1,N+1] \to \bb{R}$ which interpolates between its entries, and prove an interlacing result for the zeroes of these functions. A sign change between consecutive nonzero entries corresponds to a zero between those entries. Moreover, any zero entry of $v$ must occur between entries of opposite sign; supposing $v^i =0$, the recurrence \cref{forwardRecurrence} tells us $v^{i+1} = -\frac{b_{i-1}}{b_i} v^{i-1}$. So zeroes always correspond to a change in sign of the adjacent entries of the vector. 

\begin{prop}\label{discreteSturmPicone}Suppose $v_{i}$ and $v_{j}$  respectively solve
	\begin{equation} (\mathcal{J} - \mu_j\mathcal{Q})(v_j) = (c^1_j,0,...,0,c^2_j), \ (\mathcal{J} - \mu_i\mathcal{Q})(v_i) = (c^1_i,0,...,c^2_i)
	\end{equation}
	or equivalently, in matrix form,
	\begin{equation}\label{matrixEigenvalueProblem} (J- \mu_jQ)v_j = (c^1_j,0,...,0,c^2_j)^T, \ (J - \mu_iQ)(v_i) = (c^1_i,0,...,0,c^2_i)^T,
	\end{equation}
	with $\mu_i < \mu_j$, and $c^{1,2}_{i,j}$ are variables. Then there is a zero of $v_j$ between consecutive zeros of $v_i$.
\end{prop}
\begin{proof} Without loss of generality, suppose $v_i^{k-1}, v_i^{l+1} \leq 0$ and $v_i^k,...,v_i^l >0,$ i.e. the zeros of $v_i$ occur on the intervals $[k-1,k)$ and $(l,l+1].$ Toward a contradiction, suppose $v_j$ has no zero in between the consecutive zeros of $v_i$, and without loss of generality take $v_j^{k},...,v_j^{l+1} >0.$ Denote by $v_{i}^{k,l}$ and $v_j^{k,l}$ the subvector consisting of the $k$ through $l$ entries of $v_{i}$ and $v_j$ respectively, and $(J -\mu_{i}E)^{k,l}$ and $(J -\mu_{j}E)^{k,l}$ the corresponding block of the matrices $(J - \mu_{i}E)$ and $(J -\mu_{j}E)$ resepctively. 
	Since $J$ is tridiagonal and $v_i, v_j$ solve \cref{matrixEigenvalueProblem}, we can compute the following equations:
	\begin{equation}\begin{cases}
			(v_i^{k,l})^T(J - \mu_jQ)^{k,l}v_j^{k,l} = - v_i^k b_{k-1} v_j^{k-1} -  v_i^l b_l v_j^{l+1} \\
			(v_j^{k,l})^T(J - \mu_iQ)^{k,l}v_i^{k,l} = - v_j^k b_{k-1} v_i^{k-1} -  v_j^l b_l v_i^{l+1}.
	\end{cases}\end{equation}
	Using symmetry and taking the difference of these equations, we get 
	\begin{align}(v_i^{k,l})^T(\mu_i - \mu_j)Q^{k,l}v_j^{k,l} = b_{k-1}&(v_j^k v_i^{k-1} -v_i^kv_j^{k-1}) \\
		&+b_l(v_j^l v_i^{l+1} - v_i^l v_j^{l+1}). 
	\end{align}
	The left hand side of this equation is strictly negative. We will consider the first term $b_{k-1}(v_j^k v_i^{k-1} -v_i^kv_j^{k-1}),$ recalling from \cref{secondDifferenceOperator} that all the sub- and superdiagonal entries $b_k$ of $J$ are strictly negative.  
	\begin{itemize}
		\item{Case 1:} Suppose $v_i^{k-1} = 0$. Then we must have $v_j^{k-1}\geq 0$, since by assumption there is no zero of $v_j$ between the consecutive zeros of $v_i$. Then 
		\begin{equation}
			b_{k-1}(v_j^k v_i^{k-1} -v_i^kv_j^{k-1}) = -b_{k-1}v_i^kv_j^{k-1} \geq 0.
		\end{equation}
		\item{Case 2a:} Suppose $v_i^{k-1}<0$ and $v_j^{k-1}>0$. Then we immediately have
		\begin{equation}
			b_{k-1}(v_j^k v_i^{k-1} -v_i^kv_j^{k-1})>0.
		\end{equation}
		\item{Case 2b:} Suppose $v_i^{k-1}<0$ and $v_j^{k-1}<0$, but the zero of $v_j$ is less than or equal to that of $v_i$. More precisely, $t_j\leq t_i$, where
		\begin{equation}
			t_j = \frac{-v_j^{k-1}}{v_j^k - v_j^{k-1}} \textup{ and } t_i = \frac{-v_i^{k-1}}{v_i^k - v_i^{k-1}}. 
		\end{equation}
		Taking their difference, we have
		\begin{align} 
			0 &\geq t_j - t_i = \frac{-v_j^{k-1}}{v_j^k - v_j^{k-1}} - \frac{-v_i^{k-1}}{v_i^k - v_i^{k-1}}\\
			&\geq \frac{v_j^k v_i^{k-1} -v_i^kv_j^{k-1}}{(v_j^k - v_j^{k-1})(v_i^k-v_i^{k-1})}.
		\end{align}
		Noting that the denominator must be positive, we conclude the numerator is negative, and so 
		\begin{equation}
			b_{k-1}(v_j^k v_i^{k-1} -v_i^kv_j^{k-1}) \geq 0. 
		\end{equation}
	\end{itemize}
	By analogous case work, $b_l(v_j^l v_i^{l+1} - v_i^l v_j^{l+1})\geq 0$, contradicting our equality above.
	\renewcommand{\qedsymbol}{\lightning}
\end{proof}

\noindent Note that this theorem is not particular to a boundary value problem, it pertains to any solutions of the second order difference system $(J-\mu Q) v = (c^1,0,...,0,c^2)^T$.  Applied to the boundary value problem with Dirichlet boundary conditions, we can explicitly count the number of zeros of the eigenvectors, with the help of a few additional facts. 

\begin{prop}\label{simpleEigenvalues} The eigenvalues of the Sturm-Liouville problem with vanishing Dirichlet boundary conditions are simple. 
\end{prop}
\begin{proof} First, we claim the only solution to the Sturm-Liouville problem with vanishing Dirichlet boundary conditions satisfying $c^1 = 0$ or $c^2 = 0$ is $v = 0$; if either of $c^1$ or $c^2$ is zero, so is the other. Consider the homogeneous initial value problem $(\mathcal{J} - \mu\mathcal{Q}) = (c^1,0,...,0,c^2)$. If $c^1 = 0$, and we know from the vanishing Dirichlet boundary conditions that $v^1 = 0$, then $v^2 = 0.$ Using these as initial values in the recurrence relation \cref{forwardRecurrence}, $v^i = 0$ for all $i$, and then $c^2 = 0$. If $c^2 = 0$, we can use the exact same argument using the backward recurrence \cref{backwardRecurrence}.

	Now suppose there is a non-simple eigenvalue. Since the restriction of $J -\mu Q$ to the space of vectors with vanishing Dirichlet boundary conditions remains symmetric, there must be as many linearly independent eigenvalues as the multiplicity of $\mu$. Suppose $v$ and $v'$ are two linearly independent eigenvectors,
	\begin{align} (J -\mu Q)v &= (c^1,0,...,0,c^2)^T\\
		(J -\mu Q)v' &= (c'_1,0,...,0,c'_2)^T.
	\end{align}
	Now we compute
	\begin{equation} 
		(J-\mu Q)\left(v - \frac{c^1}{c'^1}v'\right) = \left(0,...,0,c^2- \frac{c^1}{c'^1}c'^2\right).
	\end{equation}
	However, $v - \frac{c^1}{c'^1}v'$ cannot be zero if $v$ and $v'$ are linearly independent, and so this equation contradicts our previous claim. \renewcommand{\qedsymbol}{\lightning}
\end{proof}

\begin{prop}\label{zeroCountSL} Let the eigenvalues of the Sturm-Liouville boundary value problem with vanishing Dirichlet boundary conditions be ordered by increasing eigenvalue. There is exactly one zero of the piecewise linearly interpolated eigenvector $v_{i+1}$ in between each pair of zeros of $v_i$, and $v_{1}$ has no zeros.
	Equivalently, the $i^{th}$ eigenvector has $i-1$ sign changes. 
\end{prop} 
\begin{proof} The eigenvectors $v_i$ are vectors of length $N+1$ with vanishing first and last entry, and other entries not all zero. Such vectors can have at most $N-1$ zeroes, not counting the first and last entries. Because the restriction of $J -\mu Q$ to the space of vectors with vanishing Dirichlet boundary conditions is symmetric, we have an independent set of $N-1$ eigenvectors spanning the space of vectors with vanishing Dirichlet boundary conditions. From \cref{discreteSturmPicone}, and \cref{simpleEigenvalues}, we know there is at least one zero of $v_{i+1}$ in between each zero of $v_i$, so the $(i+1)^{th}$ eigenvector has at least one more zero than the $i^{th}$. 
	
	Hence, the least number of zeros $v_i$ may have is $i-1$; if there is exactly one zero of $v_{i+1}$ in between each zero of $v_i$, and the first eigenvector $v_{1}$ has no zeros, the $i^{th}$ eigenvector has $i-1$ zeros. This is also the most number of zeros $v_i$ may have, if any eigenvector $v_{i+1}$ has more than one zero in between consecutive sign changes of $v^i$, or $v_{1}$ has any zeros, then all subsequent $v_i$ have at least $i$ zeros. Then the $(N-1)^{th}$ eigenvector has at least $N$ zeros, which is impossible. \renewcommand{\qedsymbol}{\lightning}
\end{proof}

\section{Morse Theory for Billiards} 
\subsection{The Hessian of the Length Functional}
Previously, we remarked that the notion of a critical point depends on the boundary relation we impose on our discrete Lagrangian problem. We can think of this either intrinsically, in the sense that we may think of $\mathcal{C}^N_R$ as an abstract manifold and look for critical points of a Lagrangian $L:\mathcal{C}^N_R \to \bb{R}$, or extrinsically, in the sense that we can think of $\mathcal{C}^N_R$ as a submanifold of $\mathcal{C}^N$, and look for critical points of $L:\mathcal{C}^N \to \bb{R}$ restricted to $\mathcal{C}^N_R$. Similarly, we can view the Hessian intrinisically, as a bilinear form on the abstract tangent space $T_x\mathcal{C}^N_R$, or extrinsically, as the restriction of the Hessian on $T_x\mathcal{C}^N$ to some subspace $T_x\mathcal{C}^N_R$. We will adopt the latter viewpoint, which makes it easier to compute indices with respect to different boundary conditions without having to set up intrinsic coordinates for each $\mathcal{C}^N_R$, and makes it easier to use results such as \cref{indexSubspace} to compute the index by restricting to convenient subspaces. In this case, we will mentally parenthesize $T_x\mathcal{C}^N_R$ as $(T_x\mathcal{C}^N)_R$, that is, we will think of it as a subspace of the tangent space, rather than a tangent space to a submanifold. 
We denote tangent vectors in $T_x\mathcal{C}^N$ as variations $\delta x$, with entries $\delta x^i$ corresponding to components in the natural basis $\frac{\partial}{\partial x^i}$. Tangent vectors in $T_x\mathcal{C}^N_R$ will be written extrinsically as vectors in $T_x\mathcal{C}^N$, with whatever restrictions are required of vectors in $T_x\mathcal{C}^N_R$ specified. 
\begin{dfn} The Morse index of a critical point with respect to variations in $T_x\mathcal{C}^N_R$ is denoted $i_R(x)$. We reserve the notation $i_0(x)$ for the Morse index with respect to variations in $T_x\mathcal{C}^N_0$ and the notation $i_\Delta(x)$ for the Morse index with respect to variations in $T_x\mathcal{C}^N_\Delta$.   
\end{dfn} 
With this in mind, the Hessian of $L$ at a critical point, expressed in coordinates on $T_x\mathcal{C}^N$, is a symmetric tridiagonal matrix, because $L$ is a sum of terms each of which depends only on $x^i$ and $x^{i+1}$. The diagonal entries are given by 
\begin{equation}\label{billiardHessianDiagonals}
	\begin{cases}
		a_1 = \partial_1^2 (-\ell(x^1,x^2))\\
		a_i = \partial_1^2 (-\ell(x^i,x^{i+1})) + \partial_2^2(-\ell(x^{i-1}x^i)), \ i = 2,...,N\\
		a_{N+1} = \partial_2^2(-\ell(x^N,x^{N+1}))
	\end{cases}
\end{equation}
and the sub- and super-diagonal entries are given by 
\begin{equation}\label{billiardHessianOffDiagonals}
	b_i = \partial^2_{12} (-\ell(x^i,x^{i+1}))
\end{equation}

We can calculate expressions for the entries $a_i$ and $b_i$ in terms of physical parameters of the system, namely the incidence angles $\phi^i$, the lengths of $i^{th}$ segments of the trajectory $\ell_i$, and the curvatures of the boundary at each collision point $\kappa_i$. Computing $\partial^2_1(-\ell_i)$ first, we get
\begin{align}
	\partial_1^2(-\ell_i) &= 
	\frac{-\left( \langle -\ddot{\gamma}(x^i),\gamma(x^{i+1})-\gamma(x^i)\rangle - \langle -\dot{\gamma}(x^i),-\dot{\gamma}(x^i)\rangle\right)}{\left(\langle \gamma(x^{i+1})-\gamma(x^{i}), \gamma(x^{i+1})-\gamma(x^{i})\rangle\right)^{1/2}} \\
	& \ \ \ \ \ \ \ \ \ + \frac{\left(\langle -\dot{\gamma}(x^i),\gamma(x^{i+1}) - \gamma(x^i)\rangle\right)^2}{\left(\langle \gamma(x^{i+1})-\gamma(x^{i}), \gamma(x^{i+1})-\gamma(x^{i})\rangle\right)^{3/2}}.
\end{align}
We may identify our factors of $\left(\langle \gamma(x^{i+1})-\gamma(x^{i}), \gamma(x^{i+1})-\gamma(x^{i})\rangle\right)^{1/2}$ with $\ell_i$. We also note that since $\gamma$ is a unit-speed parametrization of the boundary, we have
\begin{equation}
	\langle -\dot{\gamma}(x^i),-\dot{\gamma}(x^i)\rangle = 1
\end{equation}
as well as 
\begin{equation}
	\langle -\ddot{\gamma}(x^i),\gamma(x^{i+1})-\gamma(x^i)\rangle = \kappa_i\ell_i\sin(\phi^i)
\end{equation}
and finally,
\begin{equation}
	\langle -\dot{\gamma}(x^i),\gamma(x^{i+1}) - \gamma(x^i)\rangle = \ell_i\sin(\phi^i).
\end{equation}
Substituting these quantities in, we find
\begin{equation}
	\partial_1^2(-\ell_i) = \kappa_i\sin(\phi^i) - \frac{1}{\ell_i}\sin^2(\phi^i). 
\end{equation}
A nearly identical computation shows
\begin{equation}
	\partial_2^2(-\ell_{i-1}) = \kappa_i\sin(\phi^i) - \frac{1}{\ell_{i-1}}\sin^2(\phi^{i-1}).
\end{equation}
For $i = 2,...,N$, $a_i$ is the sum of these contributions; for $a_1$ and $a_{N+1}$ only one of the contributions is defined, so the $a_i$ are given by

\begin{equation}\label{diagonalEntries} \begin{cases}
		a_1 = \kappa_1\sin(\phi^1) - \frac{\sin^2(\phi^1)}{\ell_1} \\
		a_i = 2\kappa_i\sin(\phi^i) - \frac{\sin^2(\phi^i)}{\ell_i} - \frac{\sin^2(\phi^{i-1})}{\ell_{i-1}} \text{ for } i= 2,...,N\\
		a_N = \kappa_{N+1}\sin(\phi^{N+1}) - \frac{\sin^2(\phi^N)}{\ell_N}. 
	\end{cases}
\end{equation}  
A similar computation for the sub- and super-diagonal entries $b_i$ gives

\begin{align} 
	\partial^2_{21}(-\ell_i) &= 
	\frac{-1}{\ell_i} \sin(\phi^i)\sin(\phi^{i+1}). 
\end{align}
Moreover, since $b_i = \frac{-1}{\ell_i} \sin(\phi^i)\sin(\phi^{i+1})<0$, the Hessian is a Jacobi matrix. 

Our next proposition relates the Morse index to the number of negative eigenvalues of a discrete Sturm-Liouville problem associated to this matrix. It is essentially a discretization of the argument by Duistermaat \cite{Duistermaat:MorseIndex}. To avoid a proliferation of signs, we will again set $S_i = -\ell_i.$

\begin{prop}\label{morseIndexEigenvalueProblem} If $x \in \mathcal{C}^N_R$ is a critical point of $L$, then $i_R(x)$ is equal to the number of eigenvalues $\mu \in (-1,0)$ of the difference system
	\begin{equation}\label{variationalSL}
		\partial^2_{12}S_{i-1}\delta x^{i-1} + (\partial_1^2S_i + \partial_2^2S_{i-1} + 2\mu q) \delta x^i + \partial^2_{21}S_i\delta x^{i+1} = 0, \ i = 2,...,N
	\end{equation}
	subject to boundary conditions

	\begin{equation}\label{variationalBCs}
		\begin{cases}
			(\delta x^1,\delta x^{N+1}) \in T_{(x^1,x^{N+1})}R\\
			\big((\partial_1^2S_1+\mu q)\delta x^1 + \partial^2_{12}S_1\delta x^2, \\
			\ \ \ \ \ \ \ \  \partial^2_{21}S_N\delta x^N + (\partial_2^2S_N+\mu q)\delta x^{N+1}\big) \in(T_{(x^1,x^{N+1})}R)^\perp

		\end{cases}
	\end{equation} 
	for sufficiently large values of the parameter $q$. 
\end{prop}
\begin{proof}
	Computing the Hessian of $L(x)$, we have
	\begin{equation}\label{secondVariation}
		\begin{split} 
			H_x&(\delta x, \delta \bar{x}) = (\partial_1^2S_1\delta x^1 + \partial^2_{12}S_1\delta x^2)\delta \bar{x}^1 \\&\sum_{i=2}^N \left( \partial^2_{21}S_{i-1}\delta x^{i-1} + (\partial_1^2S_i + \partial_2^2S_{i-1})\delta x^i + \partial^2_{12}S_i\delta x^{i+1} \right)\delta \bar{x}^i \\& (\partial^2_{21} S_n \delta x^N + \partial_2^2 S_N\delta x^{N+1})\delta \bar{x}^{N+1}. 
		\end{split}
	\end{equation}  
	Next, we define a symmetric bilinear form $Q(\delta x, \delta \bar{x}) =q\delta x^1 \delta \bar{x}^1 + (\sum_{i=2}^{N-1} 2q\delta x^i \delta \bar{x}^i) + q\delta x^N \delta \bar{x}^N$ with $q$ sufficiently large, so that the corresponding matrix $(H_x + Q)$ is symmetric and positive definite. Adopting an intrinsic viewpoint, we define a linear operator $E$ on $T_x\mathcal{C}^N_R$ via
	\begin{equation}H_x(\delta x,\delta \bar{x}) = (H_x + Q)(E\delta x, \delta \bar{x}).
	\end{equation}
	This makes $E$ symmetric, as we may equate
	\begin{align}
		(H_x + Q)(E\delta x, \delta \bar{x}) &= H_x(\delta x,\delta \bar{x})\\
		&=H_x(\delta \bar{x}, \delta x) = (H_x+Q)(E\delta \bar{x}, \delta x) \\
		&=(H_x+Q)(\delta x,E\delta \bar{x}). 
	\end{align}
	Since $E$ is symmetric, it is diagonalizable, and we may take its eigenvectors as a basis. Let $\Lambda^+$ denote the span of the eigenspaces with positive or zero eigenvalues, and $\Lambda^-$ the span of the eigenspaces with strictly negative eigenvalues. Since $(H_x+Q)$ is positive definite, $H_x(\cdot,\cdot) = (H_x + Q)(E\cdot,\cdot)$ is positive semidefinite on $\Lambda^+$ and negative definite on $\Lambda^-$. 
	
	We now claim that the index of $H_x$ is in fact equal to $\dim(\Lambda^-)$, that $\Lambda^-$ is a negative definite subspace of maximal dimension for $H_x$. First, we write $T_x\mathcal{C}^N_R = \Lambda^+ \oplus \Lambda^-$, and claim that for any other subspace $L$ on which $H_x$ is negative definite, the projection to $\Lambda^-$ is injective. Indeed, if there were distinct vectors $\delta x, \ \delta \bar{x}$ in $L$ which had the same image when projected to $\Lambda^-$, their difference would lie in $\Lambda^+$ and $H_x$ would not be negative definite on $L$. So any subspace on which $H_x$ is negative definite has dimension less than or equal to $\dim(\Lambda^-)$.

	Returning to the extrinsic viewpoint, and regarding our bilinear forms and operators as matrices, we find $E\delta x = \lambda \delta x$ if and only if $(H_x - \mu Q)(\delta x, \cdot) = 0$ on $T_x\mathcal{C}^N_R$, where $\mu = \frac{\lambda}{1-\lambda}$. Hence, the dimension of the subspace of $T_x\mathcal{C}_R^N$ on which $H_x$ is negative definite is equal to the number of $\mu \in (-1,0)$ for which this problem has a solution. Writing this out longhand, we have
	\begin{equation}\label{eigenvalueProblem}
		\begin{split} 
			(H_x - \mu Q)&(\delta x, \delta \bar{x}) = ((\partial_1^2S_1 + \mu q)\delta x^1 + \partial^2_{12}S_1\delta x^2)\delta \bar{x}^1 \\&+\sum_{i=2}^N \left( \partial^2_{21}S_{i-1}\delta x^{i-1} + (\partial_1^2S_i + \partial_2^2S_{i-1} + 2\mu q)\delta x^i + \partial^2_{12}S_i\delta x^{i+1}\right)\delta \bar{x}^i \\&+ (\partial^2_{21} S_N \delta x^N + (\partial_2^2 S_N+\mu q)\delta x^{N+1})\delta \bar{x}^{N+1}. 
		\end{split}
	\end{equation}
	In order for this to vanish for all variations $\delta \bar{x} \in T_x\mathcal{C}^N_R$, we must satisfy the discrete analog of a Sturm-Liouville problem \cref{variationalSL} with boundary conditions \cref{variationalBCs}
\end{proof}

\subsection{A Morse Index Theorem for Billiards} The Morse index theorem relates the Morse index with respect to fixed-endpoint variations of a geodesic in a Riemannian manifold to the number of zeros of the Jacobi equation along the geodesic. The solutions of the Jacobi equation, Jacobi variation fields, arise precisely as derivatives of a family of variations through extremals \cite{Milnor:morseindex}. On a surface, the vanishing of the Jacobi equation can be visualized as a point of tangency between the geodesic and the envelope of a family of variations through nearby geodesics. 

Our approach is to work directly with the discrete mechanics problem, using a vector of infinitesimal displacements at each collision the discrete analog of a Jacobi field. We will show that these arise precisely as the derivative of a variation through extremals, and that the index is related to changes in sign (with respect to the orientation of the boundary) of the entries of this vector. We then identify trajectories in the discrete mechanics problem with physical straight-line trajectories in the billiard table. We can thus visualize the index by counting tangencies of a trajectory with the envelope of a variation through nearby trajectories.

\begin{dfn} Variations which solve the difference equation \cref{variationalSL} with $\mu = 0$ are called \emph{Jacobi variations}. 
\end{dfn}

\begin{prop} The orthogonal to $T_x\mathcal{C}^N_0$ in $T_x\mathcal{C}^N$ with respect to the inner product from $H_x$ consists precisely of Jacobi variations. 
\end{prop} 
\begin{proof} Regarding $H_x$ as a map to the dual space of $T_x\mathcal{C}^N_0$, the null space consists of solutions $\delta x$ such that $H_x(\delta x, \cdot)$ vanishes on $T_x\mathcal{C}^N_0$. Because variations in $T_x\mathcal{C}^N_0$ are zero in their first and last entry, the nullspace precisely consists of any solutions $\delta x \in T_x\mathcal{C}^N$ to $H_x \delta x = (c_1,0,...,0,c_2)$ with $c_1$, $c_2$ free. This is true precisely when $\delta x$ solves \cref{variationalSL}.
\end{proof}

The following two propositions prove that Jacobi variations arise precisely as variations through critical points with respect to fixed endpoint variations. 

\begin{prop} Let $x$ be a billiard trajectory, i.e. a critical point of \cref{billiardLengthFunctional} with respect to a boundary relation \cref{discreteBCs}. Suppose $x(s)$ is a $C^1$ family of critical points with respect to fixed endpoint variations parametrized on an open neighborhood of 0, with $x(0) = x$. The variation $\delta x = \dds\big|_{s = 0}x(s)$ is a Jacobi variation.
\end{prop}
\begin{proof} For each $s$, $x(s)$ is a critical point of $L$ with respect to fixed-endpoint variations. This requires $DL_{x(s)}(\delta x) = 0$ for $\delta x \in T_{x(s)}\mathcal{C}^N_0$, or equivalently that $x(s)$ solves a discrete Euler-Lagrange equation for each $s$, 
	\begin{equation}\partial_1 S(x^i(s),x^{i+1}(s)) + \partial_2 S(x^{i-1}(s),x^i(s)) = 0, i = 2,...,N.
	\end{equation}
	Differentiating with respect to $s$ and evaluating at $s = 0$, we exactly obtain \cref{variationalSL}. 
\end{proof}

\begin{prop}\label{jacobiVariationThroughCriticalTrajectories} Suppose  $\delta x$ is a Jacobi variation. There exists a $C^1$ family of trajectories, i.e. critical points with respect to fixed endpoint variations $x(s)$, with $x(0) = x$ and $\dds\big|_{s = 0}x(s) = \delta x$.
\end{prop}
\begin{proof} Since Jacobi variations solve a second-order difference equation, any Jacobi variation $\delta x$ is uniquely specified by initial values $\delta x^1$ and $\delta x^2$. We now wish to construct parametrized $x^1(s), \ x^2(s)$ satisfying $\dds\big|_{s = 0}x^1(s) = \delta x^2$, $\dds\big|_{s=0} x^2(s) = \delta x^2$, and $x^1(0) = x^1, \ x^2(0) = x^2$. Working in coordinates, we may take $x^1(s) = x^1 + s\delta x^1$ and $x^2(s) = x^2 + s\delta x^2$. Given $x^{i-1}$ and $x^i$, there is a unique $x^{i+1}$ that solves the discrete Euler-Lagrange equation for our system, so we may extend $(x^1(s), x^2(s))$ to a length-$(N+1)$ trajectory $(x^1(s),x^2(s),x^3(s),...,x^{N+1}(s))$. For each $s$, this trajectory is a critical point with respect to fixed endpoint variations, and its derivative with respect to $s$ evaluated at $s = 0$ is a Jacobi variation. Since $\dds\big|_{s = 0}x^1(s) = \delta x^1$ and $\dds\big|_{s=0}x^2(s) = \delta x^2$ uniquely determine our Jacobi variation $\delta x$, we have constructed a variation through critical points with respect to fixed endpoint variations giving rise to the desired $\delta x$. 
\end{proof}

\begin{dfn} Let $x_J(s)$ be a family of variations through extremals giving rise to a Jacobi variation $\delta x_J$ with $\delta x_J^1 = 0, \delta x_J^2>0$. Identifying the sequence of collision coordinates $x_J(s)$ with the corresponding physical trajectories in the billiard table, we call the envelope of the family $x_J(s)$ the \emph{caustic} of $x_J(s)$, and a point in the table along the physical trajectory \emph{conjugate} to $x^1$ if it is a point of tangency of $x$ with the caustic. 
\end{dfn}

\noindent See \cref{conjugatefig} for a depiction. This terminology is analogous to the definition of a caustic in classical mechanics and the study of Lagrangian singularities. We note that it collides with the usual definition of a caustic in the billiard literature, which refers to a curve in the billiard table with the property that trajectories tangent to the caustic remain tangent after reflection \cite{Tabachnikov:billiards}. 

\begin{prop}\label{geometricConjugatePoints} The Morse index $i_0(x)$ with respect to fixed-endpoint variations of a length-$N$ billiard trajectory $x$ is equal to $N-1$ minus the number of conjugate points along $x$, excluding a possible conjugate point at $x^{N+1}$. If $x^{N+1}$ is a conjugate point, then $x$ is degenerate with respect to fixed-endpoint variations, that is, its nullity is 1, otherwise $x$ is nondegenerate and its nullity is 0. 
\end{prop}

\begin{proof} In \cref{morseIndexEigenvalueProblem}, we proved that the Morse index $i_0(x)$ with respect to fixed-endpoint variations is equal to the number of eigenvalues $\mu \in (-1,0)$ of the discrete Sturm-Liouville eigenvalue problem $(H_x - \mu Q) \delta x = (c^1,0,...,0,c^2)^T$ with boundary conditions $\delta x^1 = \delta x^{N+1} = 0$. Let $\delta x_{(\mu^-)}$ be the solution with the largest negative (closest to zero) eigenvalue and and let $\delta x_{(\mu^+)}$ be the solution with the smallest nonnegative eigenvalue. 
	
	By \cref{zeroCountSL} $\delta x_{(\mu^-)}$ and $\delta x_{(\mu^+)}$ have $i_0(x) - 1$ and $i_0(x)$ sign changes respectively, not counting the first and last entry, which must be zero per the boundary conditions. 
	
	Let $\delta x_J$ be the Jacobi vector solving the Sturm-Liouville initial value problem $H_x \delta x_J = (c^1,0,...,0,c^2)^T$ with $\delta x_J^1 = 0$ and $\delta x_J^2 >0$. Since this problem corresponds to the Sturm-Liouville problem with $\mu=0$, the zeros of of $\delta x_J$ interlace those of $\delta x_{\mu^-}$ and $\delta x_{(\mu^+)}$, so $\delta x_J$ has $i_0(x)$ zeros.
	
	Geometrically, we observe that the absence of a sign change of $\delta x_J$ between the $i^{th}$ and $(i+1)^{th}$ collision indicates a tangency with the caustic along the segment of the trajectory between them. If $\delta x_J$ has $i_0(x)$ sign changes not at collision points, there are $(N-1) - i_0(x)$ segments or collision points on which no sign change happens. If there are zeros of $\delta x_J$ at a collision point, the tangency with the caustic occurs at the collision point along the boundary and the index count does not change. Moreover, there is a tangency with the caustic at the $(N+1)^{th}$ collision if and only if $\delta x_J$ is in the kernel of $D^2L_x$, hence if and only if $x$ is degenerate with respect to fixed-endpoint variations and its nullity is 1.

\end{proof}

\subsection{Morse Theory for Periodic Trajectories and their Iterates } 
Our forthcoming discussion of the stability and instability of periodic trajectories depends on an understanding of the Morse indices of periodic trajectories with respect to periodic variations, and indices of higher iterates of these trajectories.

\begin{dfn} Let $x$ be a nondegenerate critical point with respect to periodic variations. We define the \emph{concavity} of $x$, denoted $C_\Delta(x)$ to be the difference between $i_\Delta(x)$ and $i_0(x)$. That is, $C_\Delta(x)$ satisfies
	\begin{equation}i_\Delta(x) = i_0(x)+C_\Delta(x).
	\end{equation}
\end{dfn}
\noindent The term concavity is borrowed from the terminology used in \cite{BTZ:geodesics}. Looking forward, we would like to compute $i_\Delta(x)$ from $i_0(x)$ and $C_\Delta(x)$, rather than define $C_\Delta(x)$ in terms of $i_\Delta(x)$ and $i_0(x)$. Just as we had an index theorem relating $i_0(x)$ to the sign changes of a Jacobi variation, we can relate the concavity $C_\Delta(x)$ to the behavior of some special variation.

\begin{prop}\label{concavityVector} Let $x$ be a nondegenerate periodic trajectory of a convex billiard. Up to scaling, there is a unique periodic Jacobi variation $\delta x_\Delta$ solving $H_x(\delta x_\Delta, \cdot) = (c^1,0,...,0,c^2).$ The concavity is then given by
\begin{equation}
	C_\Delta(x) = \begin{cases} 1, \text{ if } H_x(\delta x_\Delta, \delta x_\Delta) \leq 0\\
		0, \text{ if } H_x(\delta x_\Delta, \delta x_\Delta) >0.
	\end{cases}
\end{equation}
\end{prop} 
\begin{proof} The bilinear form $H_x$ is nondegenerate restricted to $T_x\mathcal{C}^N_\Delta$, so 
\begin{equation}
	\dim(T_x\mathcal{C}^N_0) + \dim((T_x\mathcal{C}^N_0)^\perp) = \dim(T_x\mathcal{C}^N_\Delta)
\end{equation}
and thus $\dim((T_x\mathcal{C}^N_0)^\perp) = 1$. So there exists $\delta x_\Delta \in (T_x\mathcal{C}^N_0)^\perp$ and it is unique up to scaling. Recall that this Jacobi vector must solve
\begin{equation}
	H_x \delta x_\Delta = (c^1,0,...,0,c^2)^T,
\end{equation}
and because $H_x$ is nondegenerate on $T_x\mathcal{C}^N_\Delta$, $c^1 \neq -c^2$ and $c^1,c^2 \neq 0$. With this in mind, $H_x(\delta x_\Delta, \delta x_\Delta) = 0$ if and only if $\delta x_\Delta$ is a fixed-endpoint Jacobi variation, that is, when $x$ is degenerate with respect to fixed-endpoint variations. Conversely, $x$ is nondegenerate with respect to fixed-endpoint variations if and only if $H_x(\delta x_\Delta, \delta x_\Delta) \neq 0$. 

In either case, we compute the concavity using \cref{IndexDecomposition},
\begin{align} 
	\ind H_x\big|_{T_x\mathcal{C}^N_\Delta} = &\ind H_x\big|_{T_x\mathcal{C}^N_0} + \ind H_x\big|_{(T_x\mathcal{C}^N_0)^\perp} \\ 
	&+ \dim(T_x\mathcal{C}^N_0 \cap (T_x\mathcal{C}^N_0)^\perp) + \dim(T_x\mathcal{C}^N_0 \cap \ker H_x\big|_{T_x\mathcal{C}^N_\Delta}).	
\end{align}
The latter three terms constitute the concavity. The last term automatically vanishes because we assume $H_x$ is nondegenerate on $H_x\big|_{T_x\mathcal{C}^N_\Delta}$. 

In the case where $H_x(\delta x_\Delta, \delta x_\Delta) \neq 0$, $x$ is nondegenerate with respect to fixed-endpoint variations, so $\dim(T_x\mathcal{C}^N_0 \cap (T_x\mathcal{C}^N_0)^\perp)=0$ and the third term vanishes as well. In this case, the second term is the index of $H_x$ along the the subspace spanned by $\delta x_\Delta$, which we compute via the sign of $H_x(\delta x_\Delta,\delta x_\Delta)$. If this sign is strictly negative, $C_\Delta(x)= 1$, and if the sign is strictly positive, $C_\Delta(x) = 0$.

When $H_x(\delta x_\Delta, \delta x_\Delta) = 0$, $x$ is degenerate with respect to fixed-endpoint variations, so $\delta x_\Delta$ is a fixed-endpoint variation, and $\dim(T_x\mathcal{C}^N_0 \cap (T_x\mathcal{C}^N_0)^\perp) = 1$. Finally, $\ind H_x\big|_{(T_x\mathcal{C}^N_0)^\perp)} = 0$, so our formula gives $C_\Delta(x) = 1$. 
\end{proof}

As is the case for fixed-endpoint variations, we would like to see the concavity reflected in a geometric feature of trajectories in the billiard table. If $x$ is degenerate with respect to fixed-endpoint variations, then the periodic Jacobi vector $\delta x_\Delta$ is a fixed-endpoint Jacobi vector, so $x^{N+1}$ and $x^1$ are conjugate. Geometrically, this situation occurs if and only if the caustic of the family $x_\Delta(s)$ giving rise to $\delta x_\Delta$ is tangent to the trajectory $x$ precisely at the reflection point $x^1 = x^{N+1}$.
When $x$ is nondegenerate with respect to fixed-endpoint variations, we can deduce the concavity from a family of trajectories giving rise to periodic Jacobi variations. Such trajectories satisfy $x^1 = x^{N+1}$ but are not periodic trajectories in the dynamical systems sense, as the incidence angles $\phi^1$ and $\phi^{N+1}$ must differ. 
\begin{dfn} Let $x$ be a billiard trajectory satisfying $x^1 = x^{N+1}$, but not necessarily $\phi^1 = \phi^{N+1}$. We will call such trajectories \emph{configuration-periodic}.
\end{dfn}
We recall from \cref{concavityVector} that up to scaling, there is a unique periodic Jacobi variation $\delta x_\Delta$ at $x$, and from \cref{jacobiVariationThroughCriticalTrajectories}, there is a family $x(s)$ such that $\dds\big|_{x = 0} x(s) = \delta x_\Delta$. It remains to show that this family may be taken to be configuration-periodic. 
\begin{prop} Let $x$ be a periodic trajectory in a billiard, nondegenerate with respect to both periodic and fixed-endpoint variations. There is a $C^1$ family $x_\Delta(s)$ of configuration-periodic trajectories parametrized on an open neighborhood of 0 such that $x_\Delta(0) = x$ and $\dds\big|_{x = 0} x_\Delta(s) = \delta x_\Delta$is a periodic Jacobi variation.
\end{prop}

\begin{proof}
Let $F$ be the map taking a coordinate and incidence angle to the subsequent coordinate and incidence angle via the billiard reflection law. Consider the function $f(x,\phi) = x - \pi_x(F^N(x,\phi))$, where $\pi_x$ denotes projection onto the $x$ component of $F^N(x,\phi)$. Since $(x^1,\phi^1)$ is a periodic point, $f(x^1,\phi^1) = 0.$ 

Moreover, we claim $\frac{\partial f}{\partial \phi} \neq 0$ at $(x^1,\phi^1)$. To see this, we consider the family $x_0(s)$ of reflection points of length $N+1$ trajectories with initial condition $(x^1, \phi^1 + s)$; if $\frac{\partial f}{\partial \phi} = 0$ at $(x^1,\phi^1)$, then $\delta x_0 = \dds\big|_{s=0} x_0(s)$ is a Jacobi vector with $\delta x^1 = \delta x^{N+1} = 0$, and our trajectory is a degenerate critical point with respect to fixed endpoint variations. Incidentally, this means $x^{N+1}$ is conjugate to $x^1$. 

By the implicit function theorem, the level set $f(x,\phi) = 0$ can be written as the graph of a function $\phi(x)$ in a neighborhood of $(x^1,\phi^1)$. For any choice of $x^1(s)$ with $x^1(0) = x$, the family of trajectories with initial condition $(x^1(s),\phi^1(x(s)))$ is a Jacobi variation through configuration-periodic trajectories in a neighborhood of $s = 0$. 
\end{proof}

Next we will prove that the difference between $\phi_\Delta^{N+1}(s)$ and $\phi_\Delta^{1}(s)$ in the family $x_\Delta(s)$ giving rise to a periodic Jacobi variation determines the concavity. See \cref{concavityfig} for a depiction. 
\begin{prop}\label{geometricConcavity} Let $x$ be a periodic trajectory in a billiard, nondegenerate with respect to both periodic and fixed-endpoint variations. Let $x_\Delta(s)$ be a $C^1$ variation through configuration-periodic trajectories parametrized on an open neighborhood of 0 giving rise to a periodic Jacobi variation $\delta x_\Delta$, scaled so that $\delta x_\Delta^1 >0$. Let $\phi_\Delta(s)$ be the incidence angles with the billiard boundary along the trajectories $x_\Delta(s)$. Then for sufficiently small $\epsilon>0$, 

\begin{equation}\label{concavityCases} 
	C_\Delta(x) = \begin{cases} 0, \text{ if }  \phi_\Delta^{N+1}(\epsilon)- \phi_\Delta^1(\epsilon) >0 \\
		1, \text{ if }\phi_\Delta^{N+1}(\epsilon)- \phi_\Delta^1(\epsilon) < 0.
	\end{cases}
\end{equation} 
\end{prop}

\begin{proof}
By \cref{concavityVector} we can compute the concavity via the sign of $H_x(\delta x_\Delta,\delta x_\Delta)$. We may compute this using the variation through configuration-periodic trajectories $x_\Delta(s)$. 
\begin{align}H_x(\delta x_\Delta,\delta x_\Delta) =& \left(\partial_1^2S(x^1,x^2)\delta x_\Delta^1 + \partial^2_{21}S(x^1,x^2)\delta x_\Delta^2\right)\delta x_\Delta ^1 \\
	& + \left(\partial^2_{12}S(x^N,x^{N+1})\delta x_\Delta^N+ \partial_2^2S(x^N,x^{N+1})\delta x_\Delta^{N+1}\right)\delta x_\Delta^{N+1}
\end{align}
Next, we observe
\begin{align}
	\partial_1^2S(x^1,x^2)&\delta x_\Delta^1 + \partial^2_{21}S(x^1,x^2)\delta x_\Delta^2\\
	\ \ \ \ \  &= \partial_1^2S(x^1,x^2) \dds\bigg|_{s=0} x^1_\Delta(s) + \partial^2_{21}S(x^1,x^2)\dds\bigg|_{s=0} x^2_\Delta(s)\\
	& = \dds\bigg|_{s = 0}\partial_1 S(x_\Delta^1,x_\Delta^2) = \dds\bigg|_{s=0}\left(\cos(\phi_\Delta^1(s))\right).
\end{align}
By similar computation, 
\begin{align}
	(\partial^2_{12}S&(x^N,x^{N+1})\delta x_\Delta^N + \partial_2^2S(x^N,x^{N+1})\delta x_\Delta^{N+1}\\
	\ \ \ \ \ 
	\ \ \ \ \ &= \dds\bigg|_{s=0}\left(-\cos(\phi_\Delta^{N+1}(s))\right).
\end{align}
Scaling $\delta x_\Delta^1 = \delta x_\Delta^{N+1} = c >0$, we have
\begin{align}
	H_x(\delta x_\Delta,\delta x_\Delta) &= c\left(\dds\bigg|_{s=0}\left(\cos(\phi_\Delta^1(s))\right) + \dds\bigg|_{s=0}\left(-\cos(\phi_\Delta^{N+1}(s))\right)\right)\\
	&= c \dds\bigg|_{s = 0}\left(\cos(\phi_\Delta^1(s)) - \cos(\phi_\Delta^{N+1}(s))\right).
\end{align}

Since our original trajectory $x_\Delta(0) = x $ is periodic, $\phi_\Delta^1(0) = \phi^1 = \phi^{N+1} = \phi_\Delta^{N+1}(0)$, and $\cos(\phi^1(0)) - \cos(\phi^{N+1}(0)) = 0$. Thus, for $\epsilon >0$ sufficiently small $\dds\big|_{s = 0}\left(\cos(\phi_\Delta^1(s)) - \cos(\phi_\Delta^{N+1}(s))\right)$ has the same sign as $\cos(\phi_\Delta^1(\epsilon)) - \cos(\phi_\Delta^{N+1}(\epsilon))$. Next, note that 
\begin{align}
	\cos(\phi_\Delta^1(\epsilon)) - \cos(\phi_\Delta^{N+1}(\epsilon)) = 2\sin&\left(\frac{1}{2}\left(\phi_\Delta^{N+1}(\epsilon) - \phi_\Delta^{1}(\epsilon)\right)\right)\\ &\cdot\sin\left(\frac{1}{2}\left(\phi_\Delta^1(\epsilon) + \phi_\Delta^{N+1}(\epsilon)\right)\right).
\end{align}
Since $\phi \in (0,\pi)$, the factor $\sin\left(\frac{\phi_\Delta^1(\epsilon) + \phi_\Delta^{N+1}(\epsilon)}{2}\right)$ is positive, so the sign of our expression depends only on the sign of $\sin\left(\frac{\phi_\Delta^{N+1}(\epsilon) - \phi_\Delta^{1}(\epsilon)}{2}\right)$, which has the same sign as $\phi_\Delta^{N+1}(\epsilon) - \phi_\Delta^{1}(\epsilon)$. Thus, $H_x(\delta x_\Delta,\delta x_\Delta)>0$ and $C_\Delta(x)=0$ when $\phi_\Delta^{N+1}(\epsilon) - \phi_\Delta^{1}(\epsilon)>0$, and similarly $H_x(\delta x_\Delta,\delta x_\Delta) < 0$ and $C_\Delta(x)=1$ when $\phi_\Delta^{N+1}(\epsilon) - \phi_\Delta^{1}(\epsilon)<0$.
\end{proof}

\begin{figure}
\centering
\begin{subfigure}{0.4\textwidth}
	\centering
	\begin{tikzpicture}
		\node at (0,0) {\includegraphics[height=1.5in]{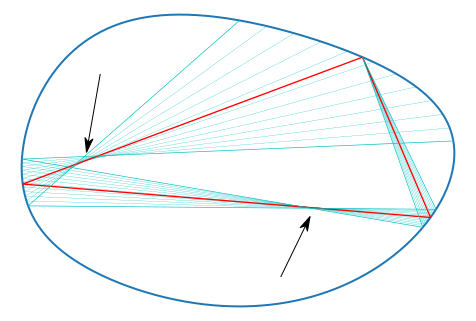}};
	\end{tikzpicture}
	\caption{\label{conjugatefig}A variation through extremals near the mountain pass 3-periodic trajectory, arrows indicate 2 conjugate points.}
\end{subfigure}%
\hspace{0.1\textwidth}
\begin{subfigure}{0.4\textwidth}
	\centering
	\begin{tikzpicture}
		\node at (0,0) {\includegraphics[height=1.5in]{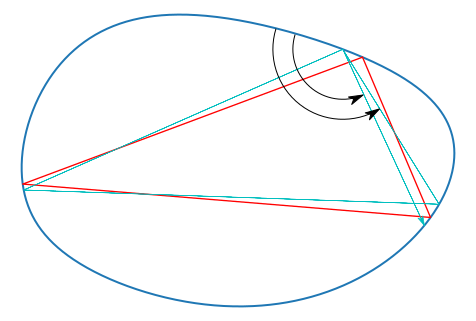}};
		\node at (.75,.27) {$\phi^1(\epsilon)$};
		\node at (1.3,1.7) {$\phi^4(\epsilon)$};
	\end{tikzpicture}
	\caption{\label{concavityfig}A configuration-periodic trajectory near the mountain pass 3-periodic trajectory, with $x^1(\epsilon)>x^1(0)$ and $\phi^4(\epsilon) -  \phi^1(\epsilon) < 0$, indicating $C_\Delta(x) = 1$.}
\end{subfigure}
\caption{\label{mtn3p1stIterateIndex}Conjugate point and concavity contributions to the Morse index for the mountain pass 3-periodic trajectory in a convex billiard. 2 conjugate points and a concavity contribution of 1 indicate index $i_\Delta(x) = (2- 2)+ 1 = 1$.}
\end{figure}
Our next goal is to write formulas for the index over two periods in terms of the index over one period. We can extend a periodic trajectory by its iterates. 

\begin{dfn} Let $x$ be an $N$-periodic trajectory in a billiard with $C^2$ boundary. Its \emph{$2^{nd}$ iterate} is the trajectory 
	\begin{equation} x^2 = (x^1,x^2,...x^{N+1},x^2,...x^{N+1}),
	\end{equation}
	that is, the trajectory $x$ traversed twice. Higher iterates are defined analogously. 
\end{dfn}

Now, let us look at the Hessian of the second iterate, and express it in terms of the Hessian over one period. 

\begin{prop}\label{secondIterateHessian} Let $x$ be an $N$-periodic trajectory, $H_x$ be the Hessian of $L$ at the critical point $x$, and $H_{x^2}$ the Hessian at its second iterate $x^2$. 
Let $\delta x^{(j,k)}$ denote the $j$ through $k^\text{th}$ entry of $\delta x$. Then 
\begin{equation}
	H_{x^2}(\delta x, \delta \bar{x}) = H_x(\delta x^{(1,N+1)}, \delta \bar{x}^{(1,N+1)}) + H_x(\delta x^{(N+1,2N+1)}, \delta \bar{x}^{(N+1,2N+1)})
\end{equation}
\end{prop}
\begin{proof} We will use the notation $a_i$ and $b_i$ from \cref{billiardHessianDiagonals} and \cref{billiardHessianOffDiagonals} over the full two-period trajectory. Observe that since the trajectory is periodic, $a_i = a_{i+N}$ for $i = 2,...N-1$, and $b_i = b_{i+N}$ for $i = 1,..,N$. Next, we introduce notation $\bar{a}_{N+1} = a_1 $ and $\underline{a}_{N+1} = a_{2N+1}.$ Note that $\bar{a}_{N+1}$ and $\underline{a}_{N+1}$ correspond to the first and last diagonal terms of the Hessian over a single period, and from \cref{billiardHessianDiagonals} it follows that $\bar{a}_{N+1} + \underline{a}_{N+1} = a_{N+1}$ in the full two-period trajectory. Thus, we compute
\begin{align}
	H_{x^2}(\delta x, \delta\bar{x}) &= (a_1\delta x^1 + b_1\delta x^2)\delta\bar{x}^1 + \sum_{i=2}^{2N} (b_{i-1}\delta x^{i-1} + a_i \delta x^i + b_{i}\delta x^{i+1})\delta \bar{x}^i\\
	& \ \ \ \ \ + (b_{2N}\delta x^{2N} + a_{2N+1}\delta x^{2N+1}) \delta \bar{x}^{2N+1}\\
	& = (a_1\delta x^1 + b_1\delta x^2)\delta\bar{x}^1 + \sum_{i=2}^{N} (b_{i-1}\delta x^{i-1} + a_i \delta x^i + b_{i}\delta x^{i+1})\delta \bar{x}^i\\
	& \ \ \ \ \ + (b_{N}\delta x^{N} + \underline{a}_{N+1}\delta x^{N+1}) \delta \bar{x}^{N+1}\\&\ \ \ \ \ + (\overline{a}_{N+1}\delta x^{N+1} + b_{N+1}\delta x^{N+2})\delta \bar{x}^{N+1}\\
	& \ \ \ \ \ + \sum_{i=N+2}^{2N} (b_{i-1}\delta x^{i-1} + a_i \delta x^i + b_{i}\delta x^{i+1})\delta \bar{x}^i\\
	& \ \ \ \ \ + (b_{2N}\delta x^{2N} + a_{2N+1}\delta x^{2N+1}) \delta \bar{x}^{2N+1}\\
	&= H_x(\delta x^{(1,N+1)}, \delta \bar{x}^{(1,N+1)}) + H_x(\delta x^{(N+1,2N+1)}, \delta \bar{x}^{(N+1,2N+1)})
\end{align}
\end{proof}

\begin{prop}\label{antiperiodicPeriodicDecomposition}
Let $x$ be a periodic trajectory of a billiard with $C^2$ boundary, and assume $x^2$ is nondegenerate with respect to periodic variations. then
\begin{equation}\label{antiperiodicPeriodicEquation}i_\Delta(x^2) = i_\Delta(x) + i_{-\Delta}(x),
\end{equation} 
where $i_{-\Delta}$ is the index of $H_x$ on 
\begin{equation} 
	T_x\mathcal{C}^N_{-\Delta} = \{\delta x \in T_x\mathcal{C}^N | \delta x^1 = -\delta x^{N+1}\}  
\end{equation}
\end{prop}
\begin{proof} We begin by defining the subspace
\begin{equation}W = \{\delta x \in T_{x^2}\mathcal{C}_\Delta^{2N} | \delta x^i = \delta x^{i+N}, \ i = 1,...N+1 \}
\end{equation}
of $T_{x^2}\mathcal{C}_\Delta^{2N}$ consisting of $N$-periodic variations, which we identify with $T_x\mathcal{C}_\Delta^N$. The complement of $W$ with respect to the form $H_{x^2}$ is 
\begin{equation}
	W^\perp = \{\delta \bar{x} \in T_{x^2}\mathcal{C}_\Delta^{2N} | H_{x^2}(\delta x, \delta \bar{x}) = 0 \textup{ for all } \delta x \in W \}.
\end{equation}
We can explicitly describe the vectors in $W^\perp$. Using the formula from proposition 2.14, 
\begin{equation}
	H_{x^2}(\delta x, \delta \bar{x}) = H_x(\delta x^{1,N+1}, \delta \bar{x}^{1,N+1} + \delta \bar{x}^{N+1,2N+1}).
\end{equation}
This vanishes for all $\delta x \in W$ only when $\delta \bar{x}^{1,N+1} + \delta \bar{x}^{N+1,2N+1} = 0$, and such $\delta \bar{x}$ comprise an $N$-dimensional subspace of the $2N$-dimensional space $T_{x^2}\mathcal{C}_\Delta^{2N}$. Since $H_{x^2}$ is nondegenerate and $W$ is also $N$-dimensional, a dimension count tells us these vectors are precisely $W^\perp$,
\begin{align}
	W^\perp &= \{\delta x \in T_{x^2}\mathcal{C}_\Delta^{2N} | \delta x^{1,N+1} + \delta x^{N+1,2N+1} = 0\}\\
	& = \{\delta x \in T_{x^2}\mathcal{C}_\Delta^{2N} | \delta x^i = -\delta x^{i+N}, \ i=1,...N+1\},
\end{align}
and immediately $W \cap W^\perp = \{0\}$. 
The space $W^\perp$ may be identified with anti-periodic variations $\delta x \in T_x\mathcal{C}^N_{-\Delta}$.  

Again using the formula from \cref{secondIterateHessian}, identifying $W$ with $T_x\mathcal{C}_\Delta^N$ and $W^\perp$ with $T_x\mathcal{C}_{-\Delta}^N$,
\begin{align}
	&H_{x^2}|_W = 2H_x|_{T_x\mathcal{C}_\Delta^N}\\
	&H_{x^2}|_{W^\perp} = 2H_x|_{T_x\mathcal{C}_{-\Delta}^N}.
\end{align}
In summary,
\begin{align} 
	&\ind H_{x^2}|_W = \ind H_x|_{T_x\mathcal{C}^N_\Delta}\\
	&\ind H_{x^2}|_{W^\perp} = \ind H_x|_{T_x\mathcal{C}^N_{-\Delta}}\\
	& \dim (W \cap W^\perp) = 0\\
	& \dim \left(W \cap \ker H_{x^2}|_{T_{x^2}\mathcal{C}^{2N}_\Delta}\right) = 0. 
\end{align}
\Cref{antiperiodicPeriodicEquation} then follows from \cref{indexSubspace}.
\end{proof} 

\noindent Note that $T_x\mathcal{C}^N_{-\Delta}$ is the space of variations about $x$ that are anti-periodic in their configuration component. As mentioned before, we are adopting the extrinsic view that this is not the tangent space to some path space submanifold $\mathcal{C}^N_{-\Delta}$, but rather a subspace of the tangent space $T_x\mathcal{C}^N$.

In a similar fashion to \cref{concavityVector}, we can further decompose $i_{-\Delta}(x)$ into a contribution from the subspace of fixed-endpoint variations, as well as a contribution from a particular anti-periodic variation. As a complement to the concavity, we will call the contribution from the anti-periodic variations the convexity. 
\begin{dfn} Let $x$ be a nondegenerate critical point with respect to anti-periodic variations. We define the \emph{convexity} of $x$, denoted $C_{-\Delta}(x)$ to be the difference between $i_{-\Delta}(x)$ and $i_0(x)$. That is, $C_{-\Delta}(x)$ satisfies
\begin{equation}i_{-\Delta}(x) = i_0(x)+C_{-\Delta}(x).
\end{equation}
\end{dfn}
Again, we can characterize this contribution in terms of a particular Jacobi variation. The proof of \cref{convexityVector} is analogous to that of \cref{concavityVector} and is omitted for brevity. %

\begin{prop}\label{convexityVector}
Let $x$ be a periodic trajectory of a convex billiard which is nondegenerate with respect to antiperiodic variations. Up to scaling, there is a unique antiperiodic Jacobi variation solving $H_x(\delta x_{-\Delta}, \cdot) = (c^1,0,...,0,c^2).$ The concavity is then given by
\begin{equation}
	C_{-\Delta}(x) = \begin{cases} 1, \text{ if } H_x(\delta x_{-\Delta}, \delta x_{-\Delta})\leq 0\\
		0, \text{ if } H_x(\delta x_{-\Delta}, \delta x_{-\Delta}) >0.
	\end{cases}
\end{equation}
\end{prop}

Finally, as an immediate corollary of \cref{antiperiodicPeriodicDecomposition}, \cref{concavityVector}, and \cref{convexityVector}, we obtain
\begin{prop}\label{indexIterationConvexityConcavity} Let $x$ be a periodic trajectory of a billiard with $C^2$ boundary. Assume $x^2$ is nondegenerate with respect to periodic variations, and further that $x$ is nondegenerate with respect to periodic and antiperiodic variations. Then 
\begin{equation}
	i_\Delta(x^2) = 2i_0(x) + C_\Delta(x) + C_{-\Delta}(x).
\end{equation} 
\end{prop}
\noindent Thus, the parity of $i_{\Delta}(x^2)$ depends entirely on the behavior of the convexity and concavity contributions. 

\subsection{Stability of Periodic Trajectories} 
Beyond billiards, there are many results relating the Morse index of a periodic trajectory of a Lagrangian dynamical system to its stability. This line of work goes at least as far back as  Poincar\'e, who proved that minimizing closed geodesics on surfaces are unstable. As a sampling of further progress, Ballman, Thorbergsson and Ziller used techniques from Morse theory to find conditions on the curvature of manifolds ensuring the existence of an elliptic-parabolic closed geodesic \cite{BTZ:geodesics}, Offin proved more general conditions under which minimizing geodesics are unstable \cite{Offin:HyperbolicGeodesics}, and Offin and Deng proved a stability alternative in low-dimensional systems using the second iterate of a periodic orbit \cite{D-O:czindex}. These developments appeal to some symplectic geometry machinery, particularly the relationship between the Morse index of a trajectory and the Maslov index of the linearized Hamiltonian flow along the lift of the trajectory to the cotangent space. In a discrete mechanics setting, MacKay and Meiss proved that periodic trajectories with index 0 have positive real reciprocal multipliers and are linearly unstable, and that periodic trajectories with index 1 have negative real reciprocal multipliers or multipliers on the unit circle \cite{MackayMeiss:LinearStability}. Finally, Kozlov and Treshchev directly relate the signature of the Hessian to the multipliers of the Poincar\'e matrix for the linearized Hamiltonian flow, and also classify stability using a second iterate criterion \cite{KozTre:billiards}. We will combine Offin and Deng's or Kozlov and Treschev's argument with MacKay and Meiss' result to prove a stability alternative in low-dimensional discrete systems using the second iterate of a periodic orbit.

First, to discuss stability, we state necessary definitions and facts from Hamiltonian mechanics and dynamical systems. 

\begin{dfn} Let $(x^i,y^i)$ be an $N$-periodic trajectory of a discrete Hamiltonian system or symplectic mapping $\Psi$. The \emph{multipliers} of the map $\Psi^N$ at $(x^0,y^0)$ are the eigenvalues of the differential $d\Psi^N$. 
\end{dfn}
The eigenvalues of $d\Psi^N$ are coordinate-independent so the multipliers are well-defined. The multipliers determine the linear stability or instability of the periodic trajectory. 
\begin{dfn} We classify the linear stability of the periodic trajectory $(x^i,y^i)$ as 
\begin{itemize}
	\item \emph{Linearly stable}, if every eigenvalue of $d\Psi^N$ has modulus less than or equal to 1,
	\item \emph{Linearly unstable}, if it is not linearly stable, or equivalently, if any eigenvalue of $d\Psi^N$ has modulus greater than 1.
\end{itemize}
\end{dfn}

Because $\Psi$ is a local symplectomorphism of an annulus, $d\Psi^N$ is a linear symplectomorphism at each periodic point. Eigenvalues of symplectic matrices occur in quadruplets; if $\lambda$ is an eigenvalue, so are $\lambda^{-1}, \ \bar{\lambda}$, and $\bar{\lambda}^{-1}$. As a consequence, a periodic billiard trajectory may be linearly stable, but cannot be linearly asymptotically stable. 

For a one degree-of-freedom Lagrangian system, or two-dimensional Hamiltonian system, the multipliers are either nonzero real reciprocals or complex conjugates on the unit circle. Using this information, MacKay and Meiss' relate the Morse index of periodic trajectory to its stability. We state their result below, specifically adapted to billiards. 

\begin{prop}\label{MacKayMeiss} Let $x$ be a nondegenerate $N$-periodic trajectory of a billiard with $C^2$ boundary. If $i_\Delta(x)$ is even, the multipliers of the billiard map $T^N$ are real and positive. If $i_\Delta(x)$ is odd, then the multipliers of the billiard map $T^N$ are real and negative or on the unit circle in the complex plane. 
\end{prop}

In their paper, MacKay and Meiss work with more general dynamical systems arising from a variational principle and only claim that if $i_\Delta(x) = 0$ the multipliers of $T$ are real and positive, and if $i_\Delta(x) = 1$, the multipliers of $T^N$ are real and negative or on the unit circle in the complex plane. They relate the multipliers to an expression depending on the sign of the determinant of the Hessian matrix at the periodic trajectory. Since the determinant of this matrix is the product of the eigenvalues, and the signs of the eigenvalues of a matrix determine the signature of the associated quadratic form, the sign of the determinant of the Hessian when $i_\Delta(x)$ is even is the same as when $i_\Delta(x) = 0$, and the sign of the determinant when $i_\Delta(x)$ is odd is the same as when $i_\Delta(x) = 1$, so the stronger statement proposed above is automatic from the original proof.

Using Offin and Deng's or Kozlov and Treshschev's argument, we can extend this result to determine the stability of a periodic trajectory from the index of its second iterate. 

\begin{prop}\label{doubleCover}Let $x$ be an $N$-periodic trajectory in a billiard with $C^2$ boundary, and suppose its second iterate $x^2$ is nondegenerate. Then if $i_\Delta(x^2)$ is respectively odd or even, the multipliers are respectively complex conjugates on the unit circle or real reciprocals. Consequently, periodic trajectory is respectively linearly stable or linearly unstable. 
\end{prop}
\begin{proof} The differential $dT^{2N}$ is the composition of the differentials $dT^N$, so the multipliers of $T^{2N}$ at $(x^0,y^0)$ are the squares of the multipliers of $T^N$ at $(x^0,y^0)$. Using \cref{MacKayMeiss}, we know that if the index is odd, the multipliers are either negative real reciprocals or complex conjugates on the unit circle; the former is impossible as negative real reciprocals cannot be the squares of real reciprocals nor complex conjugates on the unit circle, with the exception of the case in which the multipliers of $T^{2N}$ are both $-1$ and the multipliers of $T^N$ are $\pm i$. In either case, the multipliers of the second iterate and the original trajectory must be complex conjugates on the unit circle. We also know that if the index is even, the multipliers of $T^{2N}$ are positive real reciprocals, and since the trajectory is nondegenerate, they are not equal to 1; they cannot be squares of complex conjugates on the unit circle, and must be squares of negative or positive real reciprocals not equal to $\pm 1$. 
\end{proof}

\subsection{Existence of Periodic Trajectories}
We recall a result due to Birkhoff which guarantees the existence of both minimizing and mountain pass periodic trajectories in a convex billiard with $C^2$ boundary. We begin by defining the rotation number, which intuitively tells us how many times a periodic trajectory ``goes around the boundary" of the billiard table. 
\begin{dfn} Let $\gamma$ be the boundary of a $C^2$, convex billiard table, and suppose $\gamma$ has length $S$. Let $x^1,...,x^{N+1} \in \mathcal{C}^N_\Delta$ be the coordinates of an $N$-periodic billiard trajectory. For each $i$, we can write $x^{i+1} = x^i+r^i$ modulo $s$, for $r^i \in (0,S)$. We define the \emph{rotation number} $\rho$ of a periodic trajectory as $r^1+...+r^N \in \bb{Z}$. 
\end{dfn} 
\noindent Next, we sketch Birkhoff's argument, which proves the existence of many periodic trajectories. 
\begin{prop}\label{birkhoffExistence} For a given $N$, and $\rho \leq \lfloor (N-1)/2\rfloor$ coprime to $N$, there are at least two geometrically distinct $N-$periodic trajectories with winding number $\rho$. 
\end{prop}
\begin{proof} We will sketch a proof and refer to \cite{Tabachnikov:billiards} for more detail. The first step is to compute the path space of periodic trajectories, consisting of $x \in (S^1)^N$ with $x_i \neq x_{i+1}$ for all $i$, and show that each $\rho$ corresponds to a different connected component of the path space. In fact, the closure of each connected component is homeomorphic $S^1 \times [0,1]^{N-1}$, a compact manifold with boundary, and so the negative length functional has a minimum \cite{FarberTabachnikov:ConfigurationSpaces}. Moreover, the boundaries of the connected components are $x \in (S^1)^N$ with $x^i = x^{i+1}$ for at least one $i$. Geometrically, one can argue that the gradient of the negative length functional points away from the boundary; perturbing a boundary trajectory with $x^i = x^{i+1}$ so that $x^i\neq x^{i+1}$ makes the trajectory longer. Finally, we may cyclically permute the coordinates and arrive at a re-indexed copy of the same trajectory, so there are $N$ non-geometrically distinct minima. Applying \cref{mtnPassBoundary} between a pair of these, we find a mountain pass critical point corresponding to a second geometrically distinct periodic trajectory. 
\end{proof}
This argument further implies information about the Morse indices of the critical points; the minimizers are index 0, and nondegenerate mountain passes are index 1 by \cref{mtnPassIndex}. Our hope is to use \cref{doubleCover} to prove results about the stability of these trajectories. 

\subsection{Linearly Stable 2-Periodic Mountain Pass Trajectories}
Combining the preceding existence result and stability criterion, we can prove a sufficient condition for a billiard to have a linearly stable 2-periodic mountain pass trajectory, depending on the behavior of the evolute of the boundary curve. The stability computation recovers a result of Kozlov and Treshchev, although by calculating the index of the second iterate, rather than proceeding directly \cite{KozTre:billiards}. 

\begin{dfn} The \emph{evolute} of a $C^2$ plane curve is the locus of its centers of curvature, or equivalently, the envelope of the family of lines normal to the curve.
\end{dfn}

\begin{prop}\label{evoluteStability} Let $\gamma$ be a $C^2$, convex billiard table boundary. If $\gamma$ contains its evolute $\beta$, then the mountain pass 2-periodic trajectory from \cref{birkhoffExistence} is linearly stable.
\end{prop}
\begin{proof} This proof follows as a consequence of several previous results. Recall from \cref{antiperiodicPeriodicDecomposition} that $i_\Delta(x^2) = i_\Delta(x) +i_{-\Delta}(x)$. Since $x$ is a mountain pass critical point, $i_\Delta(x) = 1$, so we need only compute $i_{-\Delta}(x)$.  We note that a 2-periodic trajectory has perpendicular reflections with the boundary and that each segment of the trajectory is the same length $\ell$, so $\phi^i = \frac{\pi}{2}$ and $\ell_i = \ell$ for $i = 1,2,3$, and $\kappa_1 = \kappa_3$ by periodicity. As a result, our matrix for $H_{x},$ as computed in \cref{billiardHessianDiagonals} and \cref{billiardHessianOffDiagonals}, takes a particularly simple form, 
\begin{equation}
	H_{x} = \m{\kappa_1-1/\ell & -1/\ell& 0 \\ -1/\ell & 2\kappa_2 - 2/\ell & -1/\ell\\ 0 & -1/\ell & \kappa_1 - 1/\ell}.
\end{equation}
We want to know the index of this quadratic form when restricted to the subspace of anti-periodic variations, $T_x\mathcal{C}^2_{-\Delta}$. We can explicitly compute a matrix for the restriction of the bilinear form $H_x$ to this space, in the basis $\{e^1,e^2\}$ where $e^1 = (1,0,-1)^T,\  e^2 = (0,1,0)^T.$ Computing entrywise produces a diagonal matrix,

\begin{align}
	(H_x)|_{T_x\mathcal{C}^2_{-\Delta}} &= \m{H_x(e^1,e^1) &H_x(e^1,e^2)\\H_x(e^2,e^1)&H_x(e^2,e^2)} \\ &= 4\m{(\kappa_1-1/\ell)&0\\ 0&(\kappa_2-1/\ell)}
\end{align}
By hypothesis, the evolute is contained within the billiard boundary. Since the evolute is the locus of the centers of curvature and each segment of the trajectory is perpendicular to the boundary, this means $\ell > \frac{1}{\kappa_i}$. Taking the reciprocal gives $\kappa_i > 1/\ell$, so $(H_x)|_{T_x\mathcal{C}^2_{-\Delta}}$ is positive definite and $i_{-\Delta}(x) = 0$. We conclude that $i_\Delta(x^2) = 1$, and by \cref{doubleCover}, the trajectory is linearly stable. 
\end{proof}

\section{Acknowledgements}
The results presented here were substantially developed in the PhD. thesis of the first named author, under the supervision of the second named author, at Queen's University.

This work has been partially funded by the second named author's discovery grant from NSERC. 

\bibliographystyle{elsarticle-num.bst} 
\bibliography{bibliography.bib}

@ARTICLE{Duistermaat:MorseIndex,
   AUTHOR="J.J. Duistermaat",
   TITLE="On the {Morse} Index in Variational Calculus",
   JOURNAL="Advances in Mathematics",
   VOLUME="21",
   PAGES="173-195",
   YEAR="1976",
}

@ARTICLE{BTZ:geodesics,
	AUTHOR = "W. Ballmann and G. Thorbergsson and W. Ziller",
	TITLE = "Closed Geodesics on Positively Curved Manifolds",
	JOURNAL = "Annals of Mathematics",
	VOLUME = "116, No. 2",
	PAGES = "213-247",
	YEAR = "1982",
}

@BOOK{McD-S:IntroSympTop, 
	AUTHOR = "Dusa McDuff and Dietmar Salamon",
	TITLE = "Introduction to Symplectic Topology (3rd ed.)",
	YEAR = 2017,
	PUBLISHER = "Oxford University Press",
	ADDRESS = "Oxford",
}

@ARTICLE{D-O:czindex,
	AUTHOR = "Yanxia Deng and Daniel Offin",
	TITLE = "Stability of Periodic Orbits by {Conley-Zehnder} index theory",
	JOURNAL = "Journal of Differential Equations",
	YEAR = "25 March 2022",
	VOLUME = "314",
	PAGES = "473-490",
}

@ARTICLE{Offin:HyperbolicGeodesics,
	AUTHOR = "Daniel Offin",
	TITLE = "Hyperbolic Minimizing Geodesics",
	JOURNAL = "Transactions of the American Mathematical Society",
	YEAR = "2000",
	VOLUME = "352, No. 7",
	PAGES = "3323-3338",
}

@Book{Nicolaescu:MorseTheory,
	AUTHOR = "Liviu Nicolaescu",
	TITLE = "An Invitation to {Morse} Theory",
	YEAR = "2011",
	PUBLISHER = "Springer",
	ADDRESS = "New York",
}

@BOOK{Milnor:morseindex,
	AUTHOR = "John Milnor",
	TITLE = "Morse Theory",
	YEAR = "1973",
	PUBLISHER = "Princeton University Press",
	Address = "Princeton",
}

@BOOK{Tabachnikov:billiards,
	AUTHOR = "Serge Tabachnikov",
	TITLE = "Geometry and Billiards",
	YEAR = "2005",
	PUBLISHER = "American Mathematical Society",
}

@BOOK{Artin:GeometricAlgebra,
	AUTHOR = "Emil Artin",
	TITLE = "Geometric Algebra",
	YEAR = "1957",
	PUBLISHER = "Interscience Publishers Inc.",
	ADDRESS = "New York",
}

@BOOK{Teschl:JacobiOperators,
	AUTHOR = "Gerald Teschl",
	TITLE = "{Jacobi} Operators and Completely Integrable Nonlinear Lattices",
	YEAR = "2000",
	PUBLISHER = "American Mathematical Society",
	ADDRESS = "Providence",
}

@ARTICLE{MackayMeiss:LinearStability,
	AUTHOR = "R.S. {MacKay} and J.D. Meiss",
	TITLE = "Linear stability of periodic orbits in {Lagrangian} systems",
	JOURNAL = "Physics Letters A",
	YEAR = "10 October 1983",
	VOLUME = "98, Issue 3",
	PAGES = "92-94",
}

@ARTICLE{FarberTabachnikov:ConfigurationSpaces,
	AUTHOR = "Michael Farber and Serge Tabachnikov",
	TITLE = "Topology of cyclic configuration spaces and periodic trajectories of multi-dimensional billiards",
	JOURNAL = "Topology",
	YEAR = "2002",
	VOLUME = "41",
	PAGES = "553-589",
}

@BOOK{KozTre:billiards,
	AUTHOR = " {Valeri\u{\i}} Treschev and {Dmitri\u{\i}} Kozlov   ",
	TITLE = "Billiards: A Genetic Introduction to the Dynamics of Systems with Impacts",
	YEAR = "1991",
	PUBLISHER = "American Mathematical Society",
	Address = "Providence",
}

@ARTICLE{BialyTsod,
	AUTHOR = "Misha Bialy and Daniel Tsodikovich",
	TITLE = "Locally maximising orbits for the non-standard generating function of convex billiards and applications",
	JOURNAL = "Nonlinearity",
	VOLUME = "36, No. 3",
	YEAR = "2023",
}
\end{document}